\documentclass{article}
\usepackage{amsfonts}      
\usepackage{amsmath}
\usepackage{amsthm}
\usepackage{amssymb} 
\usepackage{cite}
\usepackage{setspace}
\usepackage{caption}
\usepackage{float}
\usepackage{graphicx}
\usepackage{marvosym}
\usepackage{pst-node}
\usepackage{stmaryrd}
\usepackage{tikz-cd}
\usepackage{multicol}

\usetikzlibrary{shapes,arrows}
\usetikzlibrary{decorations.pathreplacing}

\DeclareMathOperator{\Ass}{Ass}

\DeclareMathOperator{\Spec}{Spec}
\DeclareMathOperator{\Min}{Min}
\DeclareMathOperator{\height}{ht}
 
\DeclareMathOperator{\dep}{depth}

\theoremstyle{plain}
	\newtheorem{theorem}{Theorem}
	\newtheorem{lemma}[theorem]{Lemma}
    \newtheorem{corollary}[theorem]{Corollary}
    \newtheorem{proposition}[theorem]{Proposition}
\theoremstyle{definition}
	\newtheorem{remark}[theorem]{Remark}
    \newtheorem{example}[theorem]{Example}
\theoremstyle{remark}
    \newtheorem{defn}[theorem]{Definition}
\numberwithin{theorem}{section}

\newcommand{\ds}{\displaystyle}
\usepackage[top=1in, bottom=1in, left=1in, right=1in]{geometry}

\title{Characterization of Completions of Noncatenary Local Domains and Noncatenary Local UFDs}
\author{Chloe I. Avery, Caitlyn Booms, Timothy M. Kostolansky, S. Loepp, Alex Semendinger}
\date{\today}

\begin{document}
\maketitle

\doublespacing

\begin{abstract}
We find necessary and sufficient conditions for a complete local ring to be the completion of a noncatenary local (Noetherian) domain, as well as necessary and sufficient conditions for it to be the completion of a noncatenary local (Noetherian) unique factorization domain. We use our first result to demonstrate a large class of quasi-excellent domains that are not excellent, as well as a large class of catenary domains that are not universally catenary.  We use our second result to find a larger class of noncatenary local UFDs than was previously known, and we show that there is no bound on how noncatenary a UFD can be.  
\end{abstract}

\section{Introduction}

A ring $A$ is called \textit{catenary} if, for all pairs of prime ideals $P \subsetneq Q$ of $A$, all saturated chains of prime ideals between $P$ and $Q$ have the same length. Otherwise, it is called \textit{noncatenary}. For some time it was thought likely that noncatenary Noetherian rings did not exist. This was proven incorrect by Nagata in 1956, when he constructed a family of noncatenary local (Noetherian) integral domains in \cite{nagata1956chain}. Roughly speaking, this construction is accomplished by ``gluing together'' maximal ideals of different heights of a semilocal domain to obtain a noncatenary local domain. Nagata's result was later extended by Heitmann in \cite{heitmann1979examples}, where he shows that there is no finite bound on the ``noncatenarity'' of a local domain, in the sense that the difference in length between the longest and shortest saturated chains of prime ideals from $(0)$ to the maximal ideal can be made arbitrarily large (in fact, Heitmann's result is considerably stronger than this). It was then conjectured that all integrally closed domains are catenary, which Ogoma disproved in 1980 in \cite{ogoma} by constructing a noncatenary integrally closed domain. Furthermore, it was not until 1993 that the existence of a noncatenary Noetherian unique factorization domain was established by Heitmann in \cite{heitmann1993characterization}. We believe this is the only example of a noncatenary Noetherian UFD currently in the literature.

This paper contains two main results: we characterize the completions of noncatenary local domains and we characterize the completions of noncatenary local UFDs. The former is done essentially by ``gluing together'' associated prime ideals of a nonequidimensional complete local ring, an approach that is different than the previous methods of ``gluing together'' maximal ideals. We also use this construction to find a large class of rings that are quasi-excellent but not excellent, as well as a class of rings that are catenary but not universally catenary. Our second main result is a generalization of Theorem 10 in \cite{heitmann1993characterization}, and allows us to find many examples of noncatenary local UFDs. Our constructions also allow us to prove in a new way that there is no finite bound on the noncatenarity of a local domain, and we in fact extend this result to UFDs as well.


Throughout the paper, whenever we say a ring is \textit{local}, we mean that it is Noetherian and has a unique maximal ideal. We denote a local ring $A$ with unique maximal ideal $M$ by $(A,M)$. Whenever we refer to the completion of a local ring $(A,M)$, we mean the completion of $A$ with respect to $M$, and we denote this by $\widehat{A}$. Finally, we use $\height(I)$ to denote the height of the ideal $I$ and we say that the \textit{length} of a chain of prime ideals of the form $P_0 \subsetneq \dots \subsetneq P_{n}$ is $n$.

\section{Characterizing Completions of Noncatenary Local Domains}

\subsection{Background}

We first cite a result which will be important for both of our main theorems:

\begin{theorem}\textnormal{(\hspace{1sp}\cite[Theorem 31.6]{matsumura1989commutative})} \label{matsumura}
Let $A$ be a local ring such that $\widehat{A}$ is equidimensional. Then $A$ is universally catenary.
\end{theorem}

In particular, we will use the contrapositive: if $A$ is not universally catenary, then $\widehat{A}$ is nonequidimensional. This provides a simple necessary condition for a complete local ring $T$ to be the completion of a noncatenary local ring. 

The following theorem from \cite{lech1986method} provides necessary and sufficient conditions for a complete local ring to be the completion of a local domain. These conditions will be necessary for Theorem \ref{iff}, where we characterize completions of noncatenary local domains.

\begin{theorem}\textnormal{(\hspace{1sp}\cite[Theorem 1]{lech1986method})}\label{lech}
Let $(T,M)$ be a complete local ring. Then $T$ is the completion of a local domain if and only if the following conditions hold:
\begin{enumerate}
\item[(i)] No integer of $T$ is a zero divisor of $T$, and
\item[(ii)] Unless equal to $(0)$, $M\notin \Ass T$.
\end{enumerate} 
\end{theorem}

Our construction in the proof of Theorem \ref{iff} uses results from \cite{charters2004semilocal}. The following lemma, adapted from Lemma 2.8 in \cite{charters2004semilocal}, will be useful for pointing out additional interesting properties of the rings we construct.

\begin{lemma}\textnormal{(\hspace{1sp}\cite[Lemma 2.8]{charters2004semilocal})} \label{pippalemma}
Let $(T,M)$ be a complete local ring of dimension at least one, and let $G$ be a set of nonmaximal prime ideals of $T$ where $G$ contains the associated prime ideals of $T$ and such that the set of maximal elements of $G$ is finite. Moreover, suppose that if $Q \in \Spec T$ with $Q \subseteq P$ for some $P \in G$ then $Q \in G$. Also suppose that, for each prime ideal $P \in G$, $P$ contains no nonzero integers of $T$. Then there exists a local domain $A$ such that the following conditions hold:
\begin{enumerate}
\item[(i)] $\widehat A \cong T$,
\item[(ii)] If $P$ is a nonzero prime ideal of $A$, then $T \otimes_A k(P) \cong k(P)$, where $k(P) = A_P/PA_P$,
\item[(iii)] $\{P \in \Spec T ~|~ P \cap A = (0)\}=G$, and
\item[(iv)] If $I$ is a nonzero ideal of $A$, then $A/I$ is complete.
\end{enumerate}
\end{lemma}

\begin{remark} \label{1to1}
A particularly useful consequence of Lemma \ref{pippalemma} is that there is a one-to-one correspondence between the nonzero prime ideals of the ring $A$ and the prime ideals of $T$ that are not in $G$.  Note that the map from $\Spec T \setminus G$ to $\Spec A \setminus (0)$ is surjective since $\widehat{A}$ is a faithfully flat extension of $A$. To see that the map is injective, let $Q \in \Spec T \setminus G$ and let $P = Q \cap A$. We show that $Q = PT$. It suffices to prove that $Q/PT = PT/PT$. By (iv), $A/P$ is complete and therefore $A/P \cong \widehat{A/P} \cong T/PT$. Now observe that (letting $A/P$ denote its image in $T/PT$), we have $(Q/PT) \cap (A/P) = (Q \cap A)/P = P/P = (0)$. But since $A/P \cong T/PT$, there can only be one ideal $I$ of $T/PT$ such that $I \cap (A/P) = (0)$. Thus $Q/PT = PT/PT = (0)$ as desired. It follows that the map from $\Spec T \setminus G$ to $\Spec A \setminus (0)$ given by $Q \mapsto Q\cap A$ is bijective, with the inverse mapping given by $P \mapsto PT$. It is clear that this map is also inclusion-preserving. This result will be used heavily in the proof of Theorem \ref{iff}. $\clubsuit$
\end{remark}

The next theorem is explicitly used in our construction.

\begin{theorem}\textnormal{(\hspace{1sp}\cite[Theorem 3.1]{charters2004semilocal})} \label{pippa}
Let $(T,M)$ be a complete local ring, and $G\subseteq \Spec T$ such that $G$ is nonempty and the number of maximal elements of $G$ is finite. Then there exists a local domain $A$ such that $\widehat{A}\cong T$ and the set $\{P \in \Spec T ~|~ P \cap A = (0)\}$ is exactly the elements of $G$ if and only if $T$ is a field and $G=\{(0)\}$ or the following conditions hold:
\begin{enumerate}
\item[(i)] $M\notin G$, and $G$ contains all the associated prime ideals of $T$,
\item[(ii)] If $Q\in G$ and $P\in \Spec T$ with $P\subseteq Q$, then $P\in G$, and
\item[(iii)] If $Q\in G$, then $Q~\cap$ prime subring of $T=(0)$.
\end{enumerate}
\end{theorem}

\subsection{The Characterization}

In this section, we characterize the completions of noncatenary local domains. To do this, we start with a complete local ring $T$ and use Theorem \ref{pippa} with $G=\{P\in\Spec T ~|~ P\subseteq Q \text{ for some } Q\in\Ass T\}$ to construct a local domain $A$ such that $A$ satisfies the conditions described in Theorem \ref{pippa}.  We then use the one-to-one inclusion-preserving correspondence described in Remark \ref{1to1} to show that $A$ is noncatenary. For the reverse direction, we first need a few lemmas which describe the relationship between chains of prime ideals in a local domain and chains of prime ideals in its completion.

\begin{lemma}\label{hardinduction}
Let $A$ be a local domain such that $\widehat{A}\cong T$.  Let $M$ denote the maximal ideal of $T$, and let $\mathcal{C}_T$ be a chain of prime ideals in $T$ of the form $P_0 \subsetneq \dots \subsetneq P_{n-1} \subsetneq M$ with length $n\geq 2$ and $P_0\cap A=(0)$. If $\mathcal{C}_A$, the chain obtained by intersecting the prime ideals of $\mathcal{C}_T$ with $A$, is such that $(0)=P_0 \cap A = P_1 \cap A\subsetneq P_2\cap A\subsetneq \dots \subsetneq P_{n-1} \cap A \subsetneq M \cap A$, then $\mathcal{C}_A$ is not saturated.
\end{lemma}
\begin{proof}
We prove this by induction on $n$, the length of $\mathcal{C}_T$. If $n=2$, then $\mathcal{C}_A$ is $(0)=P_0 \cap A = P_1 \cap A \subsetneq M \cap A$. Since $\height(M \cap A)=\height M\geq 2$, there must exist a prime ideal strictly between $(0)$ and $M \cap A$. Thus $\mathcal{C}_A$ is not saturated, so the base case $n=2$ holds. Now assume that the lemma holds whenever $\mathcal{C}_T$ has length $i$ such that $2 \leq i\leq n-1$. 

We show that the lemma holds for chains of length $n$. Suppose $n\geq 3$. Then $\mathcal{C}_A$ is $(0)=P_0 \cap A =P_1\cap A \subsetneq P_2 \cap A \subsetneq \dots \subsetneq P_{n-1} \cap A \subsetneq M \cap A$. Since $P_2 \cap A \neq (0)$, we can choose a nonzero element $a \in P_2 \cap A$. Note that $a$ cannot be a zero divisor, as it is contained in the domain $A$, and it follows that $\height aT>0$. Then $\height(aT_{P_2}) > 0$ as well, so Krull's Principal Ideal Theorem gives that $\height(aT_{P_2}) = 1$. Thus $aT_{P_2}$ is contained in a height-1 prime ideal $Q' \in \Spec T_{P_2}$. Let $Q$ be the preimage of $Q'$ under the natural surjection $\Spec T \to \Spec T_{P_2}$. Then $aT \subseteq Q \subsetneq P_2$ since $\dim(T_{P_2})\geq 2$. But clearly $Q \cap A \neq (0)$, so we have that $(0) \subsetneq Q \cap A \subseteq P_2 \cap A$. 

There are two possible cases, either $Q \cap A \subsetneq P_2 \cap A$ or $Q\cap A = P_2 \cap A$ (see Figure 1). In the first case, $\mathcal{C}_A$ is not saturated. Otherwise, we consider $A' = \dfrac{A}{(Q \cap A)}$, which is also a local domain and whose completion is $T' = \dfrac{T}{(Q\cap A)T}$.

Then let $\mathcal{C}_{T'}$ be
\[\dfrac{Q}{(Q\cap A)T}\subsetneq \dfrac{P_2}{(Q\cap A)T}  \subsetneq \dots \subsetneq \dfrac{P_{n-1}}{(Q\cap A)T} \subsetneq \dfrac{M}{(Q\cap A)T}
\] 
which is a chain of prime ideals such that $\dfrac{Q}{(Q\cap A)T}\cap A'= (0)$.

Let $\mathcal{C}_{A'}$ be 
\[(0)=\dfrac{Q \cap A}{(Q\cap A)} = \dfrac{P_2 \cap A}{(Q\cap A)} \subsetneq \dots \subsetneq \dfrac{P_{n-1} \cap A}{(Q\cap A)} \subsetneq \dfrac{M \cap A}{(Q\cap A)}.
\]
Note that $\mathcal{C}_{T'}$ has length $n-1\geq 2$ because $P_2 \subsetneq M$. Then, since $\mathcal{C}_{A'}$ is of the necessary form,  our inductive hypothesis applies, so $\mathcal{C}_{A'}$ is not saturated. Therefore, $\mathcal{C}_A$ cannot be saturated. Hence, the lemma holds for chains of length $n$ in $T$, completing our inductive step and the proof.
\begin{figure}[H] 
\begin{multicols}{2}
\centering
Case 1: $\qquad$ \\
\vspace{0.15in}
\begin{tikzpicture}
\node (zero) at (0,0) {$(0)$};
\node (qa) at (-0.7,1.25) {$Q \cap A$};
\node (p2a) at (0,2.5) {$P_2 \cap A$};
	\node at (-0.9,2.5) {$a \in$};
\node at (0,3.23) {\vdots};
\node (pn-1a) at (0,3.75) {$P_{n-1} \cap A$};
\node (ma) at (0,5) {$M \cap A$};
\node at (0, 5.7) {$\underline{\Spec A}$};

\node (p0) at (3,0) {$P_0$};
\node (p1) at (3,1.25) {$P_1$};
\node (p2) at (3,2.5) {$P_2$};
\node (q) at (4,1.25) {$Q$};
	\node at (4.7, 1.25) {$\supseteq aT$};
\node at (3,3.23) {\vdots};
\node (pn-1) at (3,3.75) {$P_{n-1}$};
\node (m) at (3,5) {$M$};
\node at (3, 5.7) {$\underline{\Spec T}$};

\draw (zero) -- (qa) -- (p2a) -- (0,2.9);
\draw (0,3.37) -- (pn-1a) -- (ma);

\draw (p0) -- (p1);
\draw (p1) -- (p2) -- (3,2.9);
	\draw (q) -- (p2);
\draw (3,3.37) -- (pn-1) -- (m);

\draw [->, dashed, thick] (p1) to [out=180, in=40] (zero);
\draw [->, dashed, thick] (p0) to [out=200, in=330] (zero);
\end{tikzpicture}
\columnbreak

\centering
Case 2: \\
\vspace{0.15in}
\begin{tikzpicture}
\node (zero) at (0,0) {$(0)$};
\node (p3a) at (0,1.5) {$\dfrac{P_3 \cap A}{Q \cap A}$};
\node at (0,2.6) {\vdots};
\node (pn-1a) at (0,3.5) {$\dfrac{P_{n-1} \cap A}{Q \cap A}$};
\node (ma) at (0,5) {$\dfrac{M \cap A}{Q \cap A}$};
\node at (0, 6) {$\underline{\Spec A'}$};

\node (q) at (3,0) {$\dfrac{Q}{(Q \cap A)T}$};
\node (p2) at (3,1.5) {$\dfrac{P_2}{(Q \cap A)T}$};
\node at (3,2.6) {\vdots};
\node (pn-1) at (3,3.5) {$\dfrac{P_{n-1}}{(Q \cap A)T}$};
\node (m) at (3,5) {$\dfrac{M}{(Q \cap A)T}$};
\node at (3, 6) {$\underline{\Spec T'}$};

\draw (zero) -- (p3a) -- (0,2.2);
\draw (0,2.8) -- (pn-1a) -- (ma);

\draw (q) -- (p2) -- (3,2.2);
\draw (3,2.8) -- (pn-1) -- (m);

\draw [->, dashed, thick] (p2) to [out=190, in=50] (zero);
\draw [->, dashed, thick] (q) to [out=200, in=330] (zero);
\end{tikzpicture}
\end{multicols}
\caption{}
\end{figure}
\end{proof}

Note that the argument in the proof of Lemma \ref{hardinduction} can be generalized to the case where $(0)=P_0\cap A=P_1\cap A=\dots =P_j\cap A$ for any integer $j$ where $2\leq j\leq n-1$. Now, we will use Lemma \ref{hardinduction} to show that, in general, if $\mathcal{C}_A$ has length less than that of $\mathcal{C}_T$, then $\mathcal{C}_A$ is not saturated. 

\begin{lemma}\label{killinglongchains}
Let $A$ be a local domain such that $\widehat{A}\cong T$.  Let $M$ denote the maximal ideal of $T$, and let $\mathcal{C}_T$ be a chain of prime ideals of $T$ of the form $P_0 \subsetneq \dots \subsetneq P_{n-1} \subsetneq M$ of length  $n\geq 2$ with $P_0\cap A=(0)$. If the chain $\mathcal{C}_A$ given by $(0)=P_0 \cap A \subseteq \dots \subseteq P_{n-1} \cap A \subseteq M \cap A$ has length less than $n$, then $\mathcal{C}_A$ is not saturated.
\end{lemma}
\begin{proof}
First, note that $P_{n-1}\cap A \subsetneq M\cap A$ because $\height(P_{n-1} \cap A) \leq \height  P_{n-1} < \height M=\height(M \cap A)$. Now suppose $\mathcal{C}_A$ has length strictly less than $n$. Then there must be equality at some point in the chain, so let $m$ denote the largest integer such that $P_{m-1} \cap A = P_m \cap A$. Note that this choice of $m$ ensures that $P_m \cap A \subsetneq P_{m+1} \cap A \subsetneq \dots \subsetneq M \cap A$, and since $P_{n-1} \cap A \neq M \cap A$, $m\leq n-1$.  Let $A' = \dfrac{A}{(P_m \cap A)}$, which is also a local domain, and consider its completion $T' = \dfrac{T}{(P_m\cap A)T}$. Then let $\mathcal{C}_{T'}$ be the following chain of prime ideals  of $T'$:
\[\dfrac{P_{m-1}}{(P_m\cap A)T} \subsetneq \dfrac{P_m}{(P_m\cap A)T} \subsetneq \dfrac{P_{m+1}}{(P_m\cap A)T} \subsetneq \dots \subsetneq \dfrac{P_{n-1}}{(P_m\cap A)T} \subsetneq \dfrac{M}{(P_m\cap A)T}.
\]
Since $m \leq n-1$, this is a chain of length at least 2. Let $\mathcal{C}_{A'}$ be the corresponding chain of prime ideals of $A'$ given as follows:
\[(0)=\dfrac{P_{m-1} \cap A}{(P_m\cap A)} = \dfrac{P_m \cap A}{(P_m\cap A)} \subsetneq \dfrac{P_{m+1}\cap A}{(P_m\cap A)} \subsetneq \dots \subsetneq \dfrac{P_{n-1} \cap A}{(P_m\cap A)} \subsetneq \dfrac{M \cap A}{(P_m\cap A)}.
\]
Observe that $\mathcal{C}_{T'}$ and $\mathcal{C}_{A'}$ satisfy the conditions of Lemma \ref{hardinduction}, therefore, $\mathcal{C}_{A'}$ is not saturated. This gives us that $\mathcal{C}_A$ is not saturated.  
\end{proof}

We note that Lemma \ref{killinglongchains} is particularly useful when $\mathcal{C}_T$ is saturated. In the next lemma, we show that, in a local ring, it is possible to find saturated chains of prime ideals that satisfy nice properties.  This result will be used to prove our main theorems.

\begin{lemma} \label{chainconstruction}
Let $(T,M)$ be a local ring with $M\notin\Ass T$ and let $P\in \Min T$ with $\dim(T/P)=n$. Then there exists a saturated chain of prime ideals of $T$, $P\subsetneq Q_1\subsetneq \dots \subsetneq Q_{n-1}\subsetneq M$, such that, for each $i=1,\dots,n-1$, $Q_i \notin \Ass T$ and $P$ is the only minimal prime ideal contained in $Q_i$.
\end{lemma}
\begin{proof}
Observe that, since $\dim(T/P)=n$, there must exist a saturated chain of prime ideals in $T$ from $P$ to $M$ of length $n$, say $P\subsetneq P_1\subsetneq \dots \subsetneq P_{n-1}\subsetneq M$. We first show that we can choose $Q_1 \in \Spec T$ such that $P\subsetneq Q_1\subsetneq P_2$ is saturated, $P$ is the only minimal prime ideal of $T$ contained in $Q_1$, and $Q_1\notin \Ass T$. To do so, consider the following sets:
\begin{align*}
B &=\{Q\in \Spec T\; |\;P\subsetneq Q\subsetneq P_2 \;\text{is saturated}\}, \\
B_1 &=\{Q\in B\;|\;\exists \,P'\in \Min T\setminus \{P\} \;\mathrm{with}\; P'\subsetneq Q\}, \\
B_2 &=\{Q\in B\;|\; Q\in \Ass T\}.
\end{align*}
Then $B_1, B_2 \subseteq B$ and it suffices to find $Q_1\in B\setminus (B_1 \, \cup \, B_2)$. Note that since $T$ is Noetherian and $P\subsetneq P_1\subsetneq P_2$, we know that $B$ contains infinitely many elements and $B_2$ contains finitely many elements.

Next, we show that $B_1$ contains finitely many elements. Suppose $Q\in B_1$ contains $P'\in \Min T\setminus \{P\}$. Then $P+P'\subseteq Q$. We claim that $Q$ must be a minimal prime ideal of $P+P'$. Suppose instead that it is not. Then there must exist $Q'\in \Spec T$ with $P+P'\subseteq Q' \subsetneq Q$. But then we have $P \subsetneq Q' \subsetneq Q$ (where $P\neq Q'$ because $P' \not\subseteq P$), contradicting the fact that $P \subsetneq Q$ is saturated. Thus, if $Q\in B_1$ contains $P'$, then $Q$ is a minimal prime ideal of $P+P'$. Then, as there are only finitely many minimal prime ideals of $P+P'$ for each of the finitely many $P'\in \Min T\setminus \{P\}$, we have that $B_1$ contains finitely many elements. Therefore, we can choose $Q_1$ to be one of the infinitely many prime ideals in $B\setminus(B_1\, \cup \,B_2)$.

Now, for each $i=2,\dots , n-1$, we sequentially choose $Q_i\in \Spec T$ so that $Q_{i-1}\subsetneq Q_i\subsetneq P_{i+1}$ is saturated, $P$ is the only minimal prime ideal contained in $Q_i$, and $Q_i \notin \Ass T$ using the same argument as above. More specifically, to choose $Q_i$, redefine the set $B$ as $B =\{Q\in \Spec T\; |\;Q_{i-1}\subsetneq Q\subsetneq P_{i+1} \;\text{is saturated}\}$ and define $B_1$ and $B_2$ as before. Then $B$ is an infinite set, $B_2$ is a finite set, and we can show that $B_1$ is a finite set as above by showing that if $Q\in B_1$ contains $P'$, then $Q$ is a minimal prime ideal of $Q_{i-1}+P'$. Hence, $B\setminus(B_1\, \cup \,B_2)$ is infinite and we choose $Q_i$ to be in this set. Then the resulting chain $P\subsetneq Q_1\subsetneq \dots \subsetneq Q_{n-1}\subsetneq M$ will satisfy the desired properties.
\end{proof}

The next lemma will be used to show that the conditions in our main theorems are necessary.

\begin{lemma} \label{catoct} 
Let $(T,M)$ be a complete local ring and let $A$ be a local domain such that $\widehat{A}\cong T$. If $A$ contains a saturated chain of prime ideals from $(0)$ to $M\cap A$ of length $n$, then there exists $P\in \Min T$ such that $\dim(T/P)=n$.
\end{lemma}

\begin{proof}
Let $\mathcal{C}_A$ be a saturated chain of prime ideals in $A$ from $(0)$ to $M\cap A$ of length $n$. Since $T$ is a flat extension of $A$, we can apply the Going Down Theorem. This implies that there exists a chain of prime ideals in $T$ of length $n$ from some prime ideal $P$ to $M$, which we call $\mathcal{C}_T$, such that the image of $\mathcal{C}_T$ under the intersection map with $A$ is $\mathcal{C}_A$. We show that $P\in \Min T$ and $\mathcal{C}_T$ is saturated. To see this, suppose $P\notin \Min T$. Then there must exist some $P'\in \Min T$ such that $P'\subsetneq P$. If we extend $\mathcal{C}_T$ to contain $P'$, then this new chain will have length $n+1$ and its image under the intersection map with $A$ will also be $\mathcal{C}_A$. Then Lemma \ref{killinglongchains} implies that $\mathcal{C}_A$ is not saturated, a contradiction. So, we must have $P\in \Min T$. Additionally, by a similar argument using Lemma \ref{killinglongchains}, $\mathcal{C}_T$ is saturated. Therefore, as $T$ is catenary, $P$ is a minimal prime ideal of $T$ such that $\dim (T/P)=n$.
\end{proof}

With the above lemmas, we are now ready to prove the main theorem of this section.

\begin{theorem}\label{iff}
Let $(T,M)$ be a complete local ring. Then $T$ is the completion of a noncatenary local domain $A$ if and only if the following conditions hold:
\begin{enumerate}
\item[(i)] No integer of $T$ is a zero divisor,
\item[(ii)] $M\notin \Ass T$, and
\item[(iii)] There exists $P\in \Min T$ such that $1< \dim (T/P) < \dim T$.
\end{enumerate}
\end{theorem}

\begin{proof}
We first show that if $(T,M)$ is a complete local ring satisfying (i), (ii), and (iii), then $T$ is the completion of a noncatenary local domain $A$.

Using Theorem \ref{pippa} with $G=\{P\in \Spec T ~|~ P\subseteq Q \text{ for some } Q\in \Ass T\}$, we have that there exists a local domain $A$ whose completion is $T$ such that the set $\{P \in \Spec T ~|~ P \cap A = (0)\}$ is exactly the elements of $G$. Note that this $G$ satisfies the assumptions of Theorem \ref{pippa}. Furthermore, for this $A$, we have a one-to-one inclusion-preserving correspondence between nonzero $P\in \Spec A$ and $Q\in \Spec T\setminus G$ as described in Remark \ref{1to1}.

Let $P_0\in \Min T$ with $1 < m=\dim(T/P_0) < \dim T$, which exists by assumption. Using Lemma \ref{chainconstruction} we construct a saturated chain of prime ideals from $P_0$ to $M$ given by $P_0\subsetneq Q_1 \subsetneq Q_2 \subsetneq \dots \subsetneq Q_{m-1} \subsetneq M$, where the only minimal prime ideal contained in each $Q_i$ is $P_0$ and $Q_i \notin \Ass T$. Then, since $M\notin \Ass T$, we have that $Q_{m-1}\notin G$.

We now show that $A$ is noncatenary. Since the only minimal prime ideal contained in $Q_{m-1}$ is $P_0$, our chain is saturated, and $T$ is catenary, we have that $\height Q_{m-1} = m-1$ and $\dim(T/Q_{m-1})=1$. Then, since $Q_{m-1} \notin G$, we have $Q_{m-1} \cap A \neq (0)$. We claim that $\dim(A/(Q_{m-1} \cap A))=1$. To see this, suppose $P'\in\Spec A$ such that $Q_{m-1} \cap A \subsetneq P'$. Then by the one-to-one inclusion-preserving correspondence described in Remark \ref{1to1}, $Q_{m-1}\subsetneq P'T$. Since $Q_{m-1} \subsetneq M$ is saturated and $P'T\in \Spec T$, we must have that $P'T=M$. Therefore, $P'=P'T\,\cap\, A=M\cap A$ and it follows that $\dim(A/(Q_{m-1} \cap A))=1$. Since $\height(Q_{m-1} \cap A)\leq \height Q_{m-1}$ we have that $\height(Q_{m-1} \cap A)+\dim(A/(Q_{m-1} \cap A))\leq \height Q_{m-1}+\dim(T/Q_{m-1})=m<\dim T=\dim A$. Thus, $A$ is a noncatenary local domain whose completion is $T$.

Now, suppose that $T$ is the completion of a noncatenary local domain, $A$. The contrapositive of Theorem \ref{matsumura} implies that $T$ is nonequidimensional, and hence $\dim T=n>1$. Therefore, $T$ cannot be a field, so $M$ cannot be $(0)$. Additionally, by Theorem \ref{lech}, no integer of $T$ is a zero divisor and $M \notin \Ass T$. Since $A$ is noncatenary, there exists a saturated chain of prime ideals in $A$, call it $\mathcal{C}_A$, from $(0)$ to $M\cap A$ with length $m<n$. Since $\dim A =n>1$, we know that $(0)\subsetneq M\cap A$ is not a saturated chain. Thus, $m>1$, and consequently, $n> 2$. By Lemma \ref{catoct}, there exists $P \in \Min T$ such that $1<\dim(T/P) = m<n$, completing the proof. 
\end{proof}

\begin{remark} \label{quasi-excellent}
Let $(T,M)$ be a complete local ring satisfying conditions of (i), (ii), and (iii) of Theorem \ref{iff}, and let $A$ be the noncatenary local domain constructed in the proof of Theorem \ref{iff} whose completion is $T$. Then we claim that $A$ may, under certain circumstances, be quasi-excellent, even though it cannot be excellent. To show this, we first present the following definitions, adapted from \cite{rotthaus1997}:

\begin{defn}
A local ring $A$ is \textit{quasi-excellent} if, for all $P \in \Spec A$, the ring $\widehat{A} \otimes_A L$ is regular for every purely inseparable finite field extension $L$ of $k(P) = A_P / PA_P$. A local ring $A$ is \textit{excellent} if it is quasi-excellent and universally catenary. 
\end{defn}

To demonstrate our claim, we need to show that, for all $P \in \Spec A$ and for every purely inseparable finite field extension $L$ of $k(P)$, the ring $T \otimes_A L$ is regular. If $P \in \Spec A$ is nonzero, then, by Lemma 2.3, $T \otimes_A k(P) \cong k(P)$. Then we have $T \otimes_A L \cong T \otimes_A k(P) \otimes_{k(P)} L \cong k(P) \otimes_{k(P)} L \cong L$, a field. So in this case, $T \otimes_A L$ is regular. 

Now $A$ will be quasi-excellent if and only if $T \otimes_A L$ is regular for all purely inseparable finite field extensions $L$ of $k((0))$. For example, suppose the characteristic of $k((0))$ is zero and suppose $T_Q$ is a regular local ring for all $Q \in G$ where $G=\{P\in \Spec T ~|~ P\subseteq Q \text{ for some } Q\in \Ass T\}$. Note that $T \otimes_A k((0)) \cong S^{-1}T$, where $S = A \setminus (0)$. The prime ideals of $S^{-1}T$ are in one-to-one correspondence with the set $\{Q \in \Spec T ~|~ Q \cap A = (0)\} = G$. So, to show that $S^{-1}T$ is regular, it suffices to show that $(S^{-1}T)_Q \cong T_Q$ is a regular local ring for all $Q \in G$. But we assumed this to be true, so $A$ is quasi-excellent. Of course, $A$ cannot be excellent as it is noncatenary. $\clubsuit$
\end{remark}

We now use Remark \ref{quasi-excellent} to give a specific example of a quasi-excellent noncatenary local domain.

\begin{example} \label{noncatDomainExample}
Let $T=\dfrac{K \llbracket x,y,z,v\rrbracket}{(x)\cap (y,z)}$, where $K$ is a field of characteristic zero and $x, y, z, \text{and } v$ are indeterminates. Let $x,y,z, \text{and } v$ represent their corresponding images in $T$. Then $T$ satisfies conditions (i), (ii), and (iii) of Theorem \ref{iff} since $\Ass T = \{(x),(y,z)\}$ and $\dim(T/(y,z))=2<\dim T=3$. So, let $A$ be the noncatenary local domain constructed as in the proof of Theorem \ref{iff}. Then $G = \{Q \in \Spec T ~|~ Q \cap A = (0)\}=\Ass T$ and $T_{(x)}$ and $T_{(y,z)}$ are both regular local rings. Therefore, by Remark \ref{quasi-excellent}, $A$ is a quasi-excellent noncatenary local domain such that $\widehat{A}\cong T$.
\end{example}

In the next example we construct a class of catenary, but not universally catenary, local domains.

\begin{example} \label{catDomainExample}
Let $T=\dfrac{K \llbracket x,y_1,\dots, y_n\rrbracket}{(x)\cap (y_1,\dots, y_n)}$ where $K$ is a field, $x, y_1, \dots, y_n$ are indeterminates, and $n>1$. By Theorem \ref{lech}, we know that there exists a local domain, $A$, whose completion is $T$. Observe that $T$ contains only two minimal prime ideals, $P_1$ and $P_2$, where $\dim (T/P_1)=n$ and $\dim (T/P_2)=1$. Thus, $T$ does not satisfy condition (iii) of Theorem \ref{iff}, which implies that any such $A$ must be catenary. Additionally, Theorem 31.7 in \cite{matsumura1989commutative} states that a local ring is universally catenary if and only if $\widehat{A/P}$ is equidimensional for every $P \in \Spec A$. But since $A$ is an integral domain, we have $(0) \in \Spec A$, and $\widehat{A/(0)} = \widehat{A} \cong T$ which is nonequidimensional. Therefore, $A$ is not universally catenary.
\end{example}

\section{Characterizing Completions of Noncatenary Local UFDs}

\subsection{Background}

In this section, we find necessary and sufficient conditions for a complete local ring to be the completion of a noncatenary local unique factorization domain. Conditions (i), (ii), and (iii) of Theorem \ref{iff} will, of course, be necessary conditions. We begin by presenting a few previous results that will be useful in the proof of this section's main theorem.

The following theorem, essentially taken from \cite{heitmann1993characterization}, provides necessary and sufficient conditions for a complete local ring to be the completion of a local UFD.   

\begin{theorem}
\label{heitmann}
Let $(T,M)$ be a complete local ring. Then $T$ is the completion of a unique factorization domain if and only if it is a field, a discrete valuation ring, or it has depth at least two and no element of its prime subring is a zero divisor.
\end{theorem}

\begin{proof}
The result follows from Theorem 1 and Theorem 8 in \cite{heitmann1993characterization}.
\end{proof}

We will also use the following generalization of the Prime Avoidance Lemma to find ring elements that satisfy a certain transcendental property.

\begin{lemma}\textnormal{(\hspace{1sp}\cite[Lemma 2]{heitmann1993characterization})} \label{heitmannlemma2}
Let $(T,M)$ be a complete local ring, $C$ be a countable set of prime ideals in $\Spec T$ such that $M \notin C$ and $D$ be a countable set of elements of $T$. If $I$ is an ideal of $T$ which is contained in no single $P$ in $C$, then $I \not\subseteq \bigcup \{r+P ~|~ P \in C,~ r \in D\}$.
\end{lemma}

The construction in \cite{heitmann1993characterization} involves adjoining carefully-chosen transcendental elements to a subring while ensuring certain properties are maintained. A ring satisfying these properties is called an \textit{N-subring}, and was first defined in \cite{heitmann1993characterization}. Since we will be interested in maintaining those same properties, we present the definition of an \textit{N-subring}, where a quasi-local ring denotes a ring with one maximal ideal that is not necessarily Noetherian.

\begin{defn}
\label{Nsubring}
Let $(T,M)$ be a complete local ring and let $(R, M \cap R)$ be a quasi-local unique factorization domain contained in $T$ satisfying:
\begin{enumerate}
	\item[(i)] $|R| \leq \sup(\aleph_0, |T/M|)$ with equality only if $T/M$ is countable,
    \item[(ii)] $Q \cap R = (0)$ for all $Q \in \Ass(T)$, and
    \item[(iii)] If $t \in T$ is regular and $P \in \Ass(T/tT)$, then $\height(P \cap R) \leq 1$. 
\end{enumerate}
Then $R$ is called an \textit{N-subring} of $T$.
\end{defn}

We will also make use of the following lemma in our construction. It allows us to adjoin elements to an N-subring in such a way that the resulting ring is also an N-subring.

\begin{lemma}\textnormal{(\hspace{1sp}\cite[Lemma 11]{loepp1997constructing})} \label{loepp11}
Let $(T,M)$ be a complete local ring, $R$ be an N-subring of $T$, and $C \subset \Spec T$ such that $M \notin C$, $\Ass T \subset C$, and $\{P\in \Spec T ~|~ P\in \Ass(T/rT), 0 \neq r\in R\} \subset C$. Suppose $x\in T$ is such that, for every $P\in C$, $x\notin P$ and $x+P$ is transcendental over $R/(P\cap R)$ as an element of $T/P$. Then $S = R[x]_{(M\cap R[x])}$ is an N-subring with $R \subsetneq S$ and $|S|=\sup(\aleph_0, |R|)$.
\end{lemma}

\subsection{The Characterization}

We start by describing the main idea of our proof for characterizing completions of noncatenary local UFDs. Let $(T,M)$ be a complete local ring such that $\dep T >1$ and no integer of $T$ is a zero divisor. Our goal is to find sufficient conditions to construct a noncatenary local UFD, $A$, such that $\widehat{A}\cong T$. First, we note that if $R$ is the prime subring of $T$ localized at $M \cap R$, then $R$ is an N-subring. Now, suppose there exists $Q \in \Spec T$ such that $\dim(T/Q)=1$, $\height Q +\dim(T/Q)<\dim T$, and $\dep T_Q >1$. Then, in Lemma \ref{backwardufd}, we show that it is possible to adjoin appropriate elements of $Q$ to $R$ to obtain an N-subring, $S$, such that, if we apply the proof of Theorem 8 in \cite{heitmann1993characterization} to $S$, the resulting $A$ is a local UFD satisfying $(Q \cap A)T=Q$. We then prove that $A$ is noncatenary. Additionally, in Theorem \ref{iffufd}, we prove that our conditions are necessary.

First, we prove the following lemma, which allows us to simplify the statement of the main theorem of this section.

\begin{lemma}\label{simplifyconditions}
Let $(T,M)$ be a catenary local ring with $\dep T >1$. Then the following are equivalent:
\begin{enumerate}
\item[(i)] There exists $Q\in \Spec T$ such that $\dim(T/Q)=1$, $\height Q + \dim(T/Q)<\dim T$, and $\dep T_Q >1$.
\item[(ii)] There exists $P\in\Min T$ such that $2<\dim(T/P)<\dim T$.
\end{enumerate}
\end{lemma}
\begin{proof}
Suppose condition (i) holds for $Q\in\Spec T$, and let $P\in \Min T$ be such that $P\subseteq Q$ and $\dim(T/P) = \height Q + 1 < \dim T$. If $Q\in\Min T$, then $\dep T_Q=0$, so it must be the case that $P\subsetneq Q$. It suffices to show that $\dim(T/P)>2$. Now $\dim(T/P)>\dim(T/Q)=1$, so suppose $\dim(T/P)=2$. Then we have $\dim T_Q =1$, which implies that $\dep T_Q \leq 1$, contradicting our assumption. Therefore, $\dim(T/P)>2$. 

Now suppose condition (ii) holds for $P \in \Min T$ with $2 < \dim(T/P)=n < \dim T$. By Lemma \ref{chainconstruction}, there exists a saturated chain of prime ideals in $T$ given by $P \subsetneq Q_1 \subsetneq \dots \subsetneq Q_{n-1} \subsetneq M$ such that, for $i = 1, \dots, n-1$, $P$ is the only minimal prime ideal of $T$ contained in $Q_i$ and $Q_i \notin \Ass T$. This ensures that $\height Q_i + \dim(T/Q_i)=n<\dim T$ for each $i=1,\dots,n-1$. Note that, as a consequence of Theorem 17.2 in \cite{matsumura1989commutative}, we have $\dep T\leq \min \{\dim(T/P) ~|~ P \in \Ass T\}$. Since we have $\dep T>1$, this means that any $Q\in \Spec T$ such that $\dim(T/Q)=1$ satisfies $Q\notin \Ass T$. Therefore, $Q_{n-2}$ is not contained in any associated prime ideal of $T$, so we can find a $T$-regular element $x \in Q_{n-2}$. 

We will replace $Q_{n-1}$ in our chain with a prime ideal $Q'$ that satisfies condition (i). By the same argument used in the proof of Lemma \ref{chainconstruction}, we can now choose $Q' \in \Spec T$ such that $Q_{n-2} \subsetneq Q' \subsetneq M$ is saturated, $P$ is the only minimal prime ideal contained in $Q'$, and $Q' \notin \Ass T$. We can additionally choose $Q' \notin \Ass(T/xT)$ since $\Ass(T/xT)$ is a finite set. Observe that, since $x \in Q'$ is $T$-regular and $T_{Q'}$ is a flat extension of $T$, $x$ is $T_{Q'}$-regular.  We now find a regular element on  $T_{Q'}/xT_{Q'}$ to obtain a $T_{Q'}$-regular sequence of length 2. Since $Q' \notin \Ass(T/xT)$, by the corollary to Theorem 6.2 in \cite{matsumura1989commutative}, $Q'T_{Q'} \notin \Ass(T_{Q'}/xT_{Q'})$. Therefore, there exists $y \in Q'T_{Q'}$ which is a regular element on  $T_{Q'}/xT_{Q'}$. Thus, $x,y$ is a $T_{Q'}$-regular sequence of length 2. So, we have shown that $\dep T_{Q'} > 1$, which completes the proof. 
\end{proof}

\begin{lemma} \label{backwardufd}
Let $(T,M)$ be a complete local ring such that no integer of $T$ is a zero divisor. Suppose $\dep T >1$ and there exists $P\in \Min T$ such that $2 < \dim(T/P) < \dim T$. Then $T$ is the completion of a noncatenary local UFD.
\end{lemma}
\begin{proof}
Let $R_0$ be the prime subring of $T$ localized at its intersection with $M$, and let $C_0=\{P\in \Spec T ~|~ P \in \Ass(T/rT), 0\neq r\in R_0\}\cup \Ass T$. Note that by Lemma \ref{simplifyconditions}, there exists $Q\in \Spec T$ such that $\dim(T/Q)=1$, $\height Q + \dim(T/Q)<\dim T$, and $\dep T_Q >1$.  We first claim that $Q \notin \Ass(T/tT)$ for every $t\in T$ that is not a zero divisor. If $Q\in\Ass(T/tT)$ for some $t\in T$ that is not a zero divisor, then the corollary to Theorem 6.2 in \cite{matsumura1989commutative} gives $Q T_Q \in \Ass(T_Q/tT_Q)$. This implies that $T_Q/tT_Q$ consists of only units and zero divisors. Therefore, $t$ is a maximal regular sequence of $T_Q$. Thus, $\dep T_Q =1$, which contradicts our assumption and establishes the claim. Furthermore, if $M\in \Ass(T/tT)$ for any $t\in T$ that is not a zero divisor, then $\dep T=1$, a contradiction. Therefore, $M\notin \Ass(T/tT)$ for every $t\in T$ that is not a zero divisor and $M$ is the only prime ideal strictly containing $Q$. By the above argument, we have that $Q \not\subseteq P$ for all $P\in C_0$. Similarly, $M \not\subseteq P$ for all $P \in C_0$. Since $|R_0|\leq \aleph_0$, we have that $|C_0|\leq \aleph_0$. Then the ``countable prime avoidance lemma" \cite[Corollary 2.2]{countableprimeavoidance} gives that there exists $y_1 \in Q$ such that $y_1 \notin P$ for all $P \in C_0$. As $Q$ is finitely generated, let $Q=(x_1,\dots,x_n)$.

Next, we create a chain of N-subrings $R_0 \subsetneq R_1 \subsetneq \dots \subsetneq R_n$ so that the resulting ring, $R_n$, contains a generating set for $Q$. Note that, as in the proof of Theorem 8 of \cite{heitmann1993characterization}, $R_0$ is an N-subring. To construct our chain, at each step we replace $x_i$ with an appropriate $\tilde{x}_i$ so that $R_i=R_{i-1}[\tilde{x}_i]_{(M\cap R_{i-1}[\tilde{x}_i])}$ is an N-subring by Lemma \ref{loepp11}. Beginning with $R_0$, we find $\tilde{x}_1 = x_1+\alpha_1 y_1$ with $\alpha_1\in M$ so that $\tilde{x}_1+P$ is transcendental over $R_0/(P\cap R_0)$ as an element of $T/P$ for every $P \in C_0$. To find an appropriate $\alpha_1$, we follow an argument similar to that in Lemma 4 of \cite{heitmann1993characterization}. First, fix some $P \in C_0$ and consider $x_1+ty_1+P$ for some $t\in T$. We have $|R_0/(P\cap R_0)| \leq |R_0|$ and so the algebraic closure of $R_0/(P\cap R_0)$ in $T/P$ is countable. By Lemma 2.3 in \cite{charters2004semilocal} we have that $T/P$ is uncountable. Note that each choice of $t+P$ gives a different $x_1+ty_1+P$ since $y_1 \notin P$. So, for all but at most countably many choices of $t+P$, the image of $x_1+ty_1$ in $T/P$ will be transcendental over $R_0/(P\cap R_0)$. Let $D_{(P)}\subsetneq T$ be a full set of coset representatives of $T/P$ that make $x_1+ty_1+P$ algebraic over $R_0/(P\cap R_0)$. Let $D_0 = \ds \bigcup_{P\in C_0} D_{(P)}$. Then $|D_0|\leq \aleph_0$ since $|C_0| \leq \aleph_0$ and $|D_{(P)}|\leq \aleph_0$ for every $P\in C_0$. We can now apply Lemma \ref{heitmannlemma2} with $I=M$ to find $\alpha_1 \in M$ such that $\tilde{x}_1+P=x_1+\alpha_1 y_1+P$ is transcendental over $R_0/(P\cap R_0)$ for every $P\in C_0$. Then by Lemma \ref{loepp11}, $R_1 = R_0[\tilde{x}_1]_{(M\cap R_0[\tilde{x}_1])}$ is a countable N-subring containing $R_0$.

We now claim that $Q=(\tilde{x}_1, x_2, \dots, x_n)$. This can be seen by writing $y_1\in Q$ as $y_1=\beta_{1,1}x_1+\dots+\beta_{1,n}x_n$ for some $\beta_{1,i} \in T$. Then clearly $\tilde{x}_1\in Q$ since $x_1,y_1\in Q$ and we have
\[\tilde{x}_1=x_1+\alpha_1y_1=(1+\alpha_1\beta_{1,1})x_1+\alpha_1\beta_{1,2}x_2+\dots+\alpha_1\beta_{1,n}x_n.
\]
Rearranging gives
\[x_1 = (1+\alpha_1\beta_{1,1})^{-1}(\tilde{x}_1-\alpha_1\beta_{1,2}x_2-\dots-\alpha_1\beta_{1,n}x_n) \,\in \, (\tilde{x}_1, x_2, \dots, x_n)
\]
where $(1+\alpha_1\beta_{1,1})$ is a unit because $\alpha_1 \in M$. Thus, we can replace $x_1$ with $\tilde{x}_1$ in our generating set for $Q$.

To create $R_2$, let $C_1=\{P\in \Spec T ~|~ P \in \Ass(T/rT), 0\neq r\in R_1\}\cup \Ass T$. Then $Q\not\subseteq P$ for all $P\in C_1$. Then $|C_1|\leq \aleph_0$, so again by the ``countable prime avoidance lemma" in \cite{countableprimeavoidance}, we can find $y_2 \in Q$ such that $y_2 \notin P$ for all $P\in C_1$. Let $D_1=\ds \bigcup_{P\in C_1} D_{(P)}$ where $D_{(P)}\subsetneq T$ is a full set of coset representatives of $T/P$ that make $x_2+ty_2+P$ algebraic over $R_1/(P\cap R_1)$ for every $P \in C_1$. Then using Lemma \ref{heitmannlemma2} with $I=M$, there exists $\alpha_2 \in M$ such that $x_2+\alpha_2y_2+P$ is transcendental over $R_1/(P\cap R_1)$ for every $P \in C_1$ as an element of $T/P$. Let $\tilde{x}_2=x_2+\alpha_2 y_2$. Then $R_2 = R_1[\tilde{x}_2]_{(M\cap R_1[\tilde{x}_2])}$ is an N-subring by Lemma \ref{loepp11} and we have $Q=(\tilde{x}_1, \tilde{x}_2,x_3, \dots, x_n)$ by a similar argument as above by writing $y_2=\beta_{2,1}\tilde{x}_1+\beta_{2,2}x_2+\dots+\beta_{2,n}x_n$ to show that $x_2 \in (\tilde{x}_1, \tilde{x}_2, x_3, \dots, x_n)$. Repeating the above process for each $i=3,\dots,n$ we obtain a chain of N-subrings $R_0 \subsetneq R_1 \subsetneq \dots \subsetneq R_n$ and have $Q=(\tilde{x}_1, \tilde{x}_2, \dots, \tilde{x}_n)$. By our construction, each $\tilde{x}_i\in R_n$, so $R_n$ contains a generating set for $Q$.

In the proof of Theorem 8 in \cite{heitmann1993characterization}, Heitmann starts with a complete local ring $(T,M)$ such that no integer of $T$ is a zero divisor and $\dep T>1$.  He then takes the $N$-subring $R_0$, which, recall, is a localization of the prime subring of $T$, and constructs a local UFD containing $R_0$, whose completion is $T$. Now, to complete our construction of $A$, follow the proof of Theorem 8 in \cite{heitmann1993characterization} replacing $R_0$ with the $N$-subring $R_n$ to obtain a local UFD, $A$, such that $A$ contains $R_n$ and $\widehat{A}\cong T$.

Finally, we show that this $A$ is noncatenary. Since $R_n$ contains a generating set for $Q$ and $R_n\subsetneq A$, we have that $(Q \cap A)T = Q$. We use this and the fact that $\dim(T/Q)=1$ to show that $\dim(A/(Q\cap A))=1$. Suppose $P'$ is a prime ideal of $A$ such that $Q\cap A \subsetneq P'$. Then we have $(Q\cap A)T=Q \subsetneq P'T$. This means that the only prime ideal of $T$ that contains $P'T$ is $M$. Then $\dim(T/P'T)=0$, which implies that $\dim(A/P')=0$ since $\widehat{A/P'}=T/P'T$. It follows that $P' = M\cap A$. Thus, $\dim(A/(Q\cap A))=1$. As $\height(Q\cap A)\leq\height Q$, we have $\height(Q\cap A) + \dim(A/(Q\cap A)) \leq \height Q +\dim(T/Q) < \dim T=\dim A$. Therefore, $A$ is noncatenary.
\end{proof}

We are now prepared to prove the main theorem of this section.

\begin{theorem}\label{iffufd}
Let $(T,M)$ be a complete local ring. Then $T$ is the completion of a noncatenary local UFD if and only if the following conditions hold:
\begin{enumerate}
\item[(i)] No integer of $T$ is a zero divisor,
\item[(ii)] $\dep T>1$, and
\item[(iii)] There exists $P\in \Min T$ such that $2 < \dim(T/P) < \dim T$.
\end{enumerate}
\end{theorem}

Note that conditions (i), (ii), and (iii) immediately imply that $\dim T > 3$ and that conditions (i), (ii), and (iii) of Theorem \ref{iff} hold.

\begin{proof}
Lemma \ref{backwardufd} gives us that conditions (i), (ii), and (iii) are sufficient. We now prove that they are necessary.

Suppose $T$ is the completion of a noncatenary local UFD, $A$. Then $\dim A=n > 3$ since all local UFDs of dimension three or less are catenary. By Theorem 8 in \cite{heitmann1993characterization}, $T$ satisfies conditions (i) and (ii). Therefore, we need only show that $T$ contains a minimal prime ideal $P$ with $2<\dim(T/P)<\dim T=n$. Since $A$ is noncatenary, there exists a saturated chain of prime ideals in $A$ from $(0)$ to $M\cap A$, call it $\mathcal{C}_A$, of length $m<n$. We claim that $m > 2$. Note that $m\neq 1$ because $(0)\subsetneq M\cap A$ is not a saturated chain in $A$. So, suppose $m=2$. Then $\mathcal{C}_A$ is given by $(0)\subsetneq Q\subsetneq M\cap A$. Since $\mathcal{C}_A$ is saturated, $\height Q=1$. Since all height-1 prime ideals of a local UFD are principal, let $a\in A$ such that $Q=aA$. Now let $b\in (M\cap A)\setminus Q$ and $I=aA+bA$. Let $Q'\in \Spec A$ be a minimal prime ideal of $I$. Since $I$ is generated by two elements, Krull's Generalized Principle Ideal Theorem implies that $\height Q' <3$. Then we have $Q\subsetneq Q'\subsetneq M\cap A$ since $\height(M\cap A)=n>3$. This contradicts that $\mathcal{C}_A$ is saturated. Thus, $m>2$ as claimed. Now, by Lemma \ref{catoct}, there exists $P \in \Min T$ such that $2 < \dim(T/P) = m < n$, completing the proof.
\end{proof}

\begin{remark}
To see parallels between the above theorem and the main theorem in Section 2, it is interesting to note that condition (ii) in Theorem \ref{iff} can be replaced with the condition that $\dep T >0$ since $\dim T \geq 2$. Then Theorem \ref{iffufd} is very similar to Theorem \ref{iff} in that the only changes required are for the depth of $T$ and $\dim(T/P)$ to each increase by $1$. $\clubsuit$
\end{remark}

Note that, as a result of this theorem, given any complete local ring $T$ satisfying conditions (i), (ii), and (iii) of Theorem \ref{iffufd}, there exists a noncatenary local UFD, $A$, such that $\widehat{A} \cong T$. This allows us to show the existence of a larger class of noncatenary UFDs than was previously known, as exhibited in the example below.

\begin{example} \label{ufdexample}
Let $T=\dfrac{K\llbracket x, y_1,\dots,y_a ,z_1,\dots,z_b\rrbracket}{(x)\cap (y_1,\dots,y_a)}$, where $K$ is a field, $x,y_1,\dots,y_a,z_1,\dots,z_b$ are indeterminates, and $a$ and $b$ are integers such that $a,b>1$. Let $x,y_1,\dots,y_a,z_1,\dots,z_b$ denote their corresponding images in $T$. Note that $\dim T=a+b>3$. Then $T$ satisfies conditions (i), (ii), and (iii) of Theorem \ref{iffufd} since $\Ass T = \{(x),(y_1,\dots,y_a)\}$, $\dim(T/(y_1,\dots,y_a))=b+1<a+b=\dim T$, and $\dep T > 1$. So, we know there exists a noncatenary local UFD, $A$, such that $\widehat{A}\cong T$.
\end{example}

\subsection{Catenary Local Domains and Local UFDs}

Theorems \ref{iff} and \ref{iffufd} concern the completions of noncatenary rings, however when used in conjunction with Theorem \ref{lech} and Theorem \ref{heitmann}, we also obtain some information regarding completions of catenary local domains and catenary local UFDs. 

\begin{corollary}
Suppose $T$ is a complete local ring such that the following conditions hold:
\begin{enumerate}
\item[(i)] No integer of $T$ is a zero divisor,
\item[(ii)] $\dep T>0$, and
\item[(iii)] For all $Q\in \Min T$, either  $\dim(T/Q)\leq 1$ or $\dim (T/Q)=\dim T$.
\end{enumerate}
Then $T$ is the completion of a catenary local domain.  Moreover, every domain whose completion is $T$ is catenary.
\end{corollary}

\begin{proof}
Since $T$ is a complete local ring which satisfies (i) and (ii), Theorem \ref{lech} implies that there exists a local domain, $A$, such that $\widehat{A}\cong T$. However, by Theorem \ref{iff}, we know that $T$ is not the completion of a noncatenary local domain. Therefore, $A$ must be catenary, and every such $A$ must be catenary.
\end{proof}

\begin{corollary}\label{catenaryufd}
Suppose $T$ is a complete local ring such that the following conditions hold:
\begin{enumerate}
\item[(i)] No integer of $T$ is a zero divisor,
\item[(ii)] $\dep T>1$, and
\item[(iii)] For all $Q\in \Min T$, either  $\dim(T/Q)\leq 2$ or $\dim (T/Q)=\dim T$.
\end{enumerate}
Then $T$ is the completion of a catenary local UFD.  Moreover, every UFD whose completion is $T$ is catenary.
\end{corollary}

\begin{proof}
Since $T$ is a complete local ring which satisfies (i) and (ii), Theorem \ref{heitmann} implies that there exists a local UFD, $A$, such that $\widehat{A}\cong T$. However, by Theorem \ref{iffufd}, we know that $T$ is not the completion of a noncatenary local UFD. Therefore, $A$ must be catenary, and every such $A$ must be catenary.
\end{proof}

A consequence of these two corollaries is that there exists a class of complete local rings which are the completion of both a noncatenary local domain and a catenary local UFD. 

\begin{corollary}
Suppose $T$ is a complete local ring with $\dim T>3$ such that the following conditions hold:
\begin{enumerate}
\item[(i)] No integer of $T$ is a zero divisor,
\item[(ii)] $\dep T>1$,
\item[(iii)] For all $Q\in \Min T$, either  $\dim(T/Q)\leq 2$ or $\dim (T/Q)=\dim T$, and
\item[(iv)] There exists $P\in \Min T$ such that $\dim (T/P)=2$.
\end{enumerate}
Then $T$ is the completion of a noncatenary local domain and the completion of a catenary local UFD.
\end{corollary}

\begin{proof}
Since $T$ satisfies conditions (i), (ii), and (iv), by Theorem \ref{iff}, we know that $T$ is the completion of a noncatenary local domain. Since $T$ satisfies conditions (i), (ii), and (iii), Corollary \ref{catenaryufd} implies that $T$ is the completion of a catenary local UFD. Thus, $T$ is the completion of both a noncatenary local domain and a catenary local UFD.
\end{proof}

\section{Noncatenarity of Local Domains and Local UFDs}

As a consequence of Heitmann's main result in \cite{heitmann1979examples}, Noetherian domains can be made to be ``as noncatenary as desired,'' in the sense that, for any natural numbers $m$ and $n$, both greater than one, there exists a ring containing two prime ideals with both a saturated chain of prime ideals of length $m$ and a saturated chain of prime ideals of length $n$ between them. We reprove this result for noncatenary local domains and show that the same can be done for noncatenary local UFDs.

\begin{proposition}\label{catenaritydomain}
Let $m$ and $n$ be positive integers with $1 < m  < n$. Then there exists a noncatenary local domain of dimension $n$ with a saturated chain of prime ideals of length $m$ from $(0)$ to the maximal ideal.
\end{proposition}

\begin{proof}
Let $T$ be the complete local ring given in Example \ref{ufdexample} where $a = n-m+1$ and $b = m-1$. Observe that $a + b = \dim T$ and $1 < a < a+b$. Therefore, $T$ satisfies the conditions of Theorem \ref{iff}, and so it is the completion of a noncatenary local domain, A. By the construction of $A$ in the proof of Theorem \ref{iff}, the set $\{P \in \Spec T ~|~ P \cap A = (0)\}=\{(x),(y_1,\dots,y_a)\}=G$ and there is a one-to-one inclusion-preserving correspondence between the nonzero prime ideals of $A$ and the prime ideals of $T$ which are not in $G$. Note that $\dim (T/(x))=a+b=n$ and $\dim (T/(y_1,\dots,y_a))=b+1=m$. Therefore, there exists a saturated chain of prime ideals of $T$ from $(x)$ to $M=(x,y_1,\dots,y_a,z_1,\dots,z_b)$ of length $n$ and a saturated chain of prime ideals of $T$ from $(y_1,\dots,y_a)$ to $M$ of length $m$ (see Figure 2). By the one-to-one correspondence, the intersection map will preserve the lengths of these chains. Therefore, we have found a local domain of dimension $n$ with a saturated chain of length $m$ from $(0)$ to $M\cap A$.
\end{proof}

\begin{figure}[H]
\centering
\begin{tikzpicture} 

\node (l0) at (0,0) {$(x)$};
\node (l1) at (0,1) {$(x,z_1)$};
\node (l2) at (0,2) {$(x, z_1, z_2)$};
\node at (0,3.3) {\vdots};
\node (ltop) at (0,4.5) {$(x,z_1, \dots, z_b, y_1, \dots, y_{a-1})$};

\node (s0) at (5,1.5) {$(y_1, \dots, y_a)$};
\node (s1) at (5,2.5) {$(y_1, \dots, y_a, z_1)$};
\node at (5,3.6) {\vdots};
\node (stop) at (5,4.5) {$(y_1, \dots, y_a, z_1, \dots, z_b)$};

\node (max) at (2.5,5.75) {$M$};
\node (ring) at (2.5, 7) {$T=\dfrac{K\llbracket x, y_1,\dots,y_a ,z_1,\dots,z_b\rrbracket}{(x)\cap (y_1,\dots,y_a)}$};

\draw (l0) -- (l1) -- (l2) -- (0,2.8);
\draw (0,3.6) -- (ltop) -- (max) -- (stop) -- (5,3.8);
\draw (5,3.2) -- (s1) -- (s0);
\end{tikzpicture}
\caption{}
\end{figure}

\begin{proposition}\label{catenarityufd}
Let $m$ and $n$ be positive integers with $2 < m < n$. Then there exists a noncatenary local UFD of dimension $n$ with a saturated chain of prime ideals of length $m$ from $(0)$ to the maximal ideal.
\end{proposition}

\begin{proof}
Let $a$, $b$, and $T$ be as in the proof of Proposition \ref{catenaritydomain}. Observe that again $a+b = \dim T$ and we have $2 < a < a+b$. Furthermore, $T$ is exactly as in Example \ref{ufdexample}, so it satisfies the conditions of Theorem \ref{iffufd} and is the completion of a noncatenary local UFD, $A$. Recall that in the proof of Lemma \ref{backwardufd}, we choose a prime ideal $Q'$ of $T$ such that $\dim (T/Q')=1$ and $\height Q'+\dim (T/Q')<\dim T$ and construct $A$ such that $(Q'\cap A)T=Q'$ and $\dim(A/(Q' \cap A)) = 1$. In particular, we choose $Q'=(y_1,\dots, y_a, z_1,\dots, z_b)$, which satisfies the above (see Figure 2), and construct $A$ such that $(Q'\cap A)T=Q'$. We know that $\dim A=\dim T= a+b=n$, and we will show that $\height(Q'\cap A)=\height Q'=b$. From Theorem 15.1 in \cite{matsumura1989commutative}, since completions are faithfully flat extensions, we have that $\height Q'=\height (Q'\cap A)+\dim (T_{Q'}/(Q'\cap A)T_{Q'})$. Since $(Q'\cap A)T=Q'$, we know that $\dim (T_{Q'}/(Q'\cap A)T_{Q'})=0$, so $\height (Q'\cap A)=\height Q'$. Therefore, there exists a saturated chain of prime ideals in $A$ from $(0)$ to $M\cap A$, containing $Q'\cap A$, of length $b+1=m$. 
\end{proof}

Although we show that there is no finite bound on the noncatenarity of a local domain, as a result of Lemma \ref{catoct}, if $A$ is a local domain (or local UFD) such that $\widehat{A}\cong T$, then $A$ can only be ``as noncatenary as $T$ is nonequidimensional.'' In general, however, the converse is not true. In fact, in Example \ref{catDomainExample}, we construct a class of examples of rings which are ``as nonequidimensional as desired,'' but are not the completions of noncatenary local domains.  In other words, for any positive integer $n$, there is a complete local nonequidimensional ring $T$ with $P,Q\in \Min T$ such that $\dim(T/P)-\dim(T/Q)=n$, but every local domain $A$ such that $\widehat{A}\cong T$ must be catenary, but not universally catenary.

\subsection*{Acknowledgements}

We would like to thank Nina Pande for her many helpful conversations, as well as Williams College and the NSF for providing funding (grant \#DMS1659037). 

\bibliography{bib}{}

\begin{thebibliography}{10}

\bibitem{charters2004semilocal}
P.~Charters and S.~Loepp.
\newblock Semilocal generic formal fibers.
\newblock {\em J. Algebra}, 278(1):370--382, 2004.

\bibitem{heitmann1979examples}
Raymond~C. Heitmann.
\newblock Examples of noncatenary rings.
\newblock {\em Trans. Amer. Math. Soc.}, 247:125--136, 1979.

\bibitem{heitmann1993characterization}
Raymond~C. Heitmann.
\newblock Characterization of completions of unique factorization domains.
\newblock {\em Trans. Amer. Math. Soc.}, 337(1):379--387, 1993.

\bibitem{lech1986method}
Christer Lech.
\newblock A method for constructing bad {N}oetherian local rings.
\newblock In {\em Algebra, algebraic topology and their interactions
  ({S}tockholm, 1983)}, volume 1183 of {\em Lecture Notes in Math.}, pages
  241--247. Springer, Berlin, 1986.

\bibitem{loepp1997constructing}
S.~Loepp.
\newblock Constructing local generic formal fibers.
\newblock {\em J. Algebra}, 187(1):16--38, 1997.

\bibitem{matsumura1989commutative}
Hideyuki Matsumura.
\newblock {\em Commutative ring theory}, volume~8 of {\em Cambridge Studies in
  Advanced Mathematics}.
\newblock Cambridge University Press, Cambridge, second edition, 1989.
\newblock Translated from the Japanese by M. Reid.

\bibitem{nagata1956chain}
Masayoshi Nagata.
\newblock On the chain problem of prime ideals.
\newblock {\em Nagoya Math. J.}, 10:51--64, 1956.

\bibitem{ogoma}
Tetsushi Ogoma.
\newblock Noncatenary pseudogeometric normal rings.
\newblock {\em Japan. J. Math. (N.S.)}, 6(1):147--163, 1980.

\bibitem{rotthaus1997}
Christel Rotthaus.
\newblock Excellent rings, {H}enselian rings, and the approximation property.
\newblock {\em Rocky Mountain J. Math.}, 27(1):317--334, 1997.

\bibitem{countableprimeavoidance}
R.~Y. Sharp and P.~V\'amos.
\newblock Baire's category theorem and prime avoidance in complete local rings.
\newblock {\em Arch. Math. (Basel)}, 44(3):243--248, 1985.

\end{thebibliography}
\bibliographystyle{plain}

\end{document}